\documentclass[
12pt, 
a4paper, 
oneside, 
headinclude,footinclude, 
BCOR5mm, 
]{article}

\usepackage[
nochapters, 
beramono, 
eulermath,
pdfspacing, 
dottedtoc 
]{classicthesis} 

\usepackage{arsclassica} 

\usepackage[T1]{fontenc} 

\usepackage[utf8]{inputenc} 

\usepackage{graphicx} 

\usepackage{enumitem} 

\usepackage{amsmath,amssymb,amsthm} 

\usepackage{varioref} 

\hypersetup{
colorlinks=true, breaklinks=true, bookmarks=true,bookmarksnumbered,
urlcolor=webbrown, linkcolor=RoyalBlue, citecolor=webgreen, 
pdftitle={}, 
pdfauthor={\textcopyright}, 
pdfsubject={}, 
pdfkeywords={}, 
pdfcreator={pdfLaTeX}, 
pdfproducer={LaTeX with hyperref and ClassicThesis} 
} 

\usepackage{bm}						
\usepackage{a4wide} 
\usepackage{multirow}
\usepackage{mathabx} 
\usepackage{stmaryrd}
\usepackage[greek,english]{babel}

\newtheorem{thm}{Theorem}[section]
\newtheorem{lem}[thm]{Lemma}
\newtheorem{prop}[thm]{Proposition}

\newtheorem{coro}[thm]{Corollary}

\theoremstyle{definition}				

\newcommand{\R}{\mathbb{R}}

\newcommand{\Z}{\mathbb{Z}}

\title{\normalfont\spacedallcaps{Randomisation in the Josephus Problem}} 

\author{\spacedlowsmallcaps{Faustin ADICEAM\textsuperscript{1}, Steven ROBERTSON\textsuperscript{2},}\\ \spacedlowsmallcaps{Victor SHIRANDAMI\textsuperscript{2} \& Ioannis TSOKANOS\textsuperscript{3}}} 

\date{} 

\begin{document}

\maketitle


\renewcommand{\sectionmark}[1]{\markright{\spacedlowsmallcaps{#1}}} 
\lehead{\mbox{\llap{\small\thepage\kern1em\color{halfgray} \vline}\color{halfgray}\hspace{0.5em}\rightmark\hfil}} 

\pagestyle{scrheadings} 

\begin{flushright}
\textit{To Angela James}
\end{flushright} 

\vspace{2mm}

\begin{abstract}
\noindent The Josephus problem is a well--studied elimination problem consisting in determining the position of the survivor after repeated applications of a deterministic rule removing one person at a time from a given group. \\

\noindent  A natural probabilistic variant of this process is introduced in this paper. More precisely, in this variant, the survivor is determined after performing a succession of Bernoulli trials with parameter $p$ designating at each time which person is removed. When the number of participants tends to infinity, the main result characterises the limit distribution of the position of the survivor  with an increasing degree of precision as the parameter approaches the unbiased case $p=1/2$. Then, the convergence rate to the position of the survivor is obtained in the form of a Central Limit Theorem. \\

\noindent  A number of other variants of the suggested probabilistic elimination process are also considered. They each admit a specific limit behavior which, in most cases, is stated in the form of an  open problem.  
\end{abstract}

\tableofcontents

\let\thefootnote\relax\footnotetext{\textsuperscript{1} {Laboratoire d’analyse et de mathématiques appliquées (LAMA), Université Paris-Est Créteil, Créteil, France,} \texttt{faustin.adiceam@u-pec.fr}}
\let\thefootnote\relax\footnotetext{\textsuperscript{2}Department of Mathematics, The University of Manchester, United-Kingdom, \texttt{ steven.robertson@manchester.ac.uk} \& \texttt{ victor.shirandami@manchester.ac.uk}}
\let\thefootnote\relax\footnotetext{\textsuperscript{3}Departmento de Mathem\'atica, Instituto de Biociências, Letras e Ciências Exatas, Universidade Estadual Paulista, São José do Rio Preto, São Paulo, Brasil, \texttt{ioannis.tsokanos@unesp.br}. \\}

\section{Introduction}

The Josephus problem  originates from Flavius Josephus' (c.~AD 37 -- c.~100) recollection of the siege of Yodfat (AD 67) in his \emph{Wars of the Jews}. He relates how he was left as the only survivor among  40 besieged fellow soldiers in an elimination process aiming at not surrendering to the Romans~:

\begin{quotation}
\noindent « Since we all are resolved to die, let us rely on fate to decide the order in which we must kill each other~: the first of us that fortune will designate shall fall under a stab from the next one, and thus fate will successively mark the victims and the murderers, exempting us from attempting on our lives with our own hands. For it would be unfair if, after the others had killed themselves, there were someone who could change his feelings and would want to survive\textsuperscript{4}\footnote{\textsuperscript{4}\emph{In} \textgreek{Ἱστορία Ἰουδαϊκοῦ πολέμου πρὸς Ῥωμαίους}, III, 8, \S 7. The translation is from the authors.} ».
\end{quotation}

\noindent This story gave birth to an intriguing mathematical problem already studied by Euler in the XVIII$^{th}$ century and relying on the following interpretation of the elimination process described by Flavius Josephus~: 

\paragraph{\textbf{Rule for the Josephus Elimination  Process.}} \emph{Assume that $N\ge 1$ persons enumerated from $0$ to $N-1$ stand in a circle.  Starting from the first one, each person eliminates his right neighbour and passes the knife onto the next person still alive on his right.}\\

\noindent The problem is then to determine the position $a_N\in \left\llbracket 0, N-1\right\rrbracket$ of the survivor, where given reals $x\le y$, the shorthand notation $\left\llbracket x, y\right\rrbracket$ is used for $\left[ x, y\right]\cap\Z$. Setting for convenience $b_N=a_N+1$ when $N\ge 1$, it is well--known, see~\cite[\S~1.3]{GRK} and also the very nice Numberphile episode~\cite{numbphile}, that 
\begin{itemize}
\item[(a)] the sequence $\left(b_N\right)_{N\ge 1}$ satisfies the recurrence formula \begin{equation}\label{recjos}
b_N\;=\; 2b_{\lfloor N/2\rfloor }- (-1)^N\qquad \textrm{with}\qquad b_1=1.
\end{equation}
This relation implies that the sequence is 2--regular in the language of~\cite{AS}.
\item[(b)] the sequence $\left(b_N\right)_{N\ge 1}$ admits the closed-form expression $$b_N\;=\; 2l+1 \qquad \textrm{when}\qquad N\;=\; 2^m+l \qquad \textrm{with}\qquad 0\le l<2^m.$$
\item[(c)] the sequence $\left(b_N\right)_{N\ge 1}$ can be expressed in binary base as $$b_N\;=\; \overline{c_{k-1}\cdots c_0c_k}^2\qquad \textrm{when} \qquad N\;=\;\overline{c_{k}c_{k-1}\cdots c_0}^2.$$
\end{itemize}

\noindent  It does not seem, however, to have been noticed that the recurrence relation~\eqref{recjos} yields an elegant closed-form formula for the corresponding generating series, namely
\begin{equation*}\label{detpowser}
\sum_{N=0}^{\infty} b_Nx^N\;=\; 1+\frac{1}{1-x}\cdot\left(\frac{3x-1}{1-x}-\sum_{N=1}^{\infty}2^N x^{2^N}\right).
\end{equation*}
As this will not be needed in what follows, the verification of this identity is left to the reader.\\

\noindent Assume throughout that $N\ge 3$ is an integer and that $p\in [0,1]$ is a real parameter. The present paper is concerned with a probabilistic variant of the above (deterministic) Josephus problem. In order to state it, it is convenient to rescale the position of the participants by requiring that they should all stand on a unit circle.

\paragraph{\textbf{Rule for the Probabilistic Elimination Process.}}\emph{Let there be $N$ participants enumerated from $0$ to $N-1$, standing on a unit circle with a regular spacing between them and labelled counterclockwise. The $0^{\textrm{th}}$  participant holds first the knife~: with probability $p$ he eliminates Participant 1 (standing on his right) and with probability $1-p$ Participant $N-1$ (standing on his left) before passing the knife onto the next participant still alive in the direction of the stabbing (namely, 2 or N-2, respectively). The next participant holding the knife then stabs on the same direction as in the previous step with probability $p$ and, with probability $1-p$, in the opposite direction.} \\


\noindent The case $p=1$ recovers the classical deterministic process whereas, when $0<p<1$, the above rule can be seen as an elimination process where the survivor is determined by a succession of Bernoulli trials with parameter $p$. To put it differently, the survivor is then determined after performing a succession  coin--tossings biaised according to the value of $p$.\\



\noindent The problem is now to determine the limit distribution of the survivor as the number $N$ of participants tends to infinity.  To this end, define $g_N(n,p)$ as the probability of survival of the person labelled $n\in \left\llbracket 0, N-1\right\rrbracket$ in a round with $N$ participants when the elimination process follows Bernoulli trials with parameter $p\in (0,1)$. Let then $X_N^{(p)}$ denote a random variable on the space $\R/\Z$ (identified throughout with the unit circle) standing for the normalized position of the survivor; that is, the probability that $X_N^{(p)}=n/N$ is given by  $g_N(n,p)$.

\begin{thm}\label{mainthm}
Assume that $p\in (0,1)$. Then, the sequence of random variables $\left(X_N^{(p)}\right)_{N\ge 3}$ converges in law to a random variable supported in $\R/\Z$. The distribution of the limit random variable can be described with an increasing degree of precision as the parameter $p$ approaches the unbiased case $p=1/2$ as follows~:
\begin{itemize}
\item[1.] in the general case when $p\in (0,1)$, the convergence in law $X_N^{(p)}\overset{\mathcal{L}}{\underset{N\rightarrow\infty}{\longrightarrow}} X^{(p)}$ is verified, where $X^{(p)}$  is a Bernoulli random variable taking values in the set $\{0, 1/2\}$;

\item[2.] in the case that $p$ lies in the middle interval $(1/3, 2/3)$, the convergence in probability\textsuperscript{5}\footnote{\textsuperscript{5}Recall that convergence in probability is equivalent to convergence in law in the case that the limit is a constant.} $X_N^{(p)}\underset{N\rightarrow\infty}{\longrightarrow} 1/2$ holds;

\item[3.] in the unbiased case where $p=1/2$, a rate of convergence to the constant $1/2$ can be obtained in the form of a Central Limit Theorem. Indeed, setting for simplicity $X_N=X_N^{(1/2)}$ and assuming that the random variables $\left(X_N\right)_{N\ge 1}$ are successively drawn independently,  the following convergence in law to the standard normal distribution $\mathcal{N}(0,1)$ is then satisfied~: 
\begin{equation*}
\frac{1}{B_L}\cdot\sum_{N=1}^{L}\left(X_N-\frac{1}{2}\right)\;\overset{\mathcal{L}}{\underset{L\rightarrow\infty}{\longrightarrow}}\; \mathcal{N}(0,1).
\end{equation*}
Here, $$B_L\;=\; \sqrt{\sum_{N=3}^L \mathbb{V}_N(X_N)},$$where $\mathbb{V}_N(X_N)$ is the variance of the random variable $X_N$. This quantity can be asymptotically estimated as
\begin{equation*}
B_L\;\asymp\; \sqrt{\ln L}.
\end{equation*}
\end{itemize}
\end{thm}

\noindent The notation $B_L\asymp \sqrt{\ln L}$ means the existence of constants $C>c>0$ such that for all $L\ge 1$ large enough, $c\le B_L/ \sqrt{\ln L}\le C$. \\

\noindent Numerical simulations displayed in the appendix show that, in the notations of the theorem, the convergence in probability $X_N^{(p)}\underset{N\rightarrow\infty}{\longrightarrow} 1/2$ should actually hold for \emph{any} value of the parameter $p\in (0,1)$. The difficulty in proving this claim in Case (1) of the theorem is outlined in the final Section~\ref{finalsec}. \\

\noindent It should be noted that the convergence in  law in the probabilistic case $p\in (0,1)$ stands in sharp contrast with the deterministic situation $p=1$. Indeed, it is then easy to deduce from any of the above points (a), (b) or (c) that the quantity $a_N$ normalised by the factor of $1/N$ (so as to lie in the unit interval) does not admit a limit as $N$ tends to infinity. In the deterministic case $p=0$ however, it is not hard to see that the convergence in probability to the constant $1/2$ still holds.

\paragraph{Organisation of the paper.} In Section~\ref{sec1}, a set of recursion relations providing the probabilities of survival in a round with $N$ participants  is established as functions of  probabilities of survival in a round with $N-1$ participants. This crucially intervenes in the proof of Theorem~\ref{mainthm}  which Section~\ref{sec2} is devoted to. The final Section~\ref{finalsec} deals with  open problems related to other probabilistic variants of the Josephus problem.

\paragraph{Acknowledgments.} The authors would like to thank the referee for the outstanding quality of his report. The current version of the paper contains proofs which in some places have been considerably simplified thanks to his suggestions.

\noindent The second and third--named authors were supported by the Heilbronn Institute of Mathematical Research through the UKRI grant~: Additional Funding for Mathematical Sciences (EP/521917/1). The last named author  was supported by the FAPESP Grant 2023/06371-2.

\section{Recurrence Formulae for the Probabilities of Survival}\label{sec1}

To facilitate the expression of the probabilities of survival, the domain of the map $n\in\left\llbracket 0, N-1\right\rrbracket \;\mapsto\; g_N(n,p)$ is extended to $\Z$ by evaluating its argument modulo $N\ge 3$. Thus, for instance, $g_N(-1, p)= g_N(N-1, p)$. In what follows, it will always be understood that the integer $n$ is taken modulo $N$ before applying the formulae.

\begin{prop}[Recursion relations for the probabilities of survival.]\label{recurrel} \sloppy Let $N\ge 3$ and $p\in (0,1)$. Whenever $N\ge 4$, the probability vector $\left(g_N(n, p)\right)_{0\le n\le N-1}$ meets the recurrence relation
\begin{equation*}
g_N(n, p) = \begin{cases}
  g_{N-1}(-1, p)  &  \text{ if }  n \equiv 0 \pmod{N}  \\
  (1-p)\cdot g_{N-1}(-2, p) & \text{ if } n \equiv 1 \pmod{N}\\
  p\cdot g_{N-1}(-2, p)  & \text{ if } n\equiv -1\pmod{N} \\
  p \cdot g_{N-1}(n-2, p) + (1-p)\cdot  g_{N-1}(N-n-2, p)  &  \text{ otherwise. }  
\end{cases}
\end{equation*}
with base case $\left(g_3(0,p), g_3(1,p), g_3(2,p)\right)=(0, 1-p, p)$.
\end{prop}

\begin{proof}
The base case $N=3$ is easily verified by hand. Assume therefore that $N\ge 4$. The recurrence relation is obtained from the analysis of the consequences of the first step of the elimination process (carried out by the $0^{\textrm{th}}$ Participant) on the disposition of the remaining participants. To this end, consider the distinction of cases induced by the statement~:

\begin{itemize}
\item[$\bullet$]\emph{Probability of survival $g_{N}(0, p)$ of participant $n \equiv 0 \pmod{N}$.} After the first stabbing, with probability $p$, participant 1 is out and participant 2 holds the knife. Consequently, the  $0^{th}$ person who initiated the process with $N$ participants becomes participant $-1 \pmod{N-1}$ in a new round with $N-1$ participants~: his probability of survival is then $g_{N-1}(-1, p)$. Similarly, with probability $1-p$, participant $-1 \pmod{N}$ is out and participant $-2 \pmod{N}$ holds the knife. Since the direction of stabbing is reversed compared with the previous situation, the  $0^{th}$ person who initiated the process with $N$ participants still becomes participant $-1 \pmod{N-1}$ in a  new round with $N-1$ participants~: his probability of survival is thus again $g_{N-1}(-1, p)$. Putting these two situations together, one obtains that $$g_{N}(0, p)\;=\; p\cdot g_{N-1}(-1, p)+(1-p)\cdot g_{N-1}(-1, p)\;=\; g_{N-1}(-1, p).$$

\item[$\bullet$]\emph{Probability of survival $g_{N}(1, p)$ of participant $n \equiv 1 \pmod{N}$.} After the first step of the elimination process, with probability $p$, participant 1 is out~: his survival probability vanishes. With probability $1-p$, participant $-1 \pmod{N}$ is out and, since the direction of stabbing is reversed compared with the previous situation, participant 1 now carries label $-2 \pmod{N-1}$ in the new round with $N-1$ participants. This yields that $$g_{N}(1, p)\;=\; p\cdot 0+(1-p)\cdot g_{N-1}(-2, p)\;=\; (1-p)\cdot g_{N-1}(-2, p).$$

\item[$\bullet$]\emph{Probability of survival $g_{N}(-1, p)$ of participant $n \equiv -1 \pmod{N}$.} This case is \emph{mutatis mutandis} analogous to the previous one in such a way that $$g_{N}(-1, p)\;=\; p\cdot g_{N-1}(-2, p) +(1-p)\cdot 0\;=\; p\cdot g_{N-1}(-2, p).$$

\item[$\bullet$]\emph{Probability of survival $g_{N}(n, p)$ of participant $n \not\equiv 0, \pm 1 \pmod{N}$.} In this case, after the first elimination, with probability $p$, Participant $n$ carries the label $n-2$ and his survival probability in the new round with $N-1$ participants becomes $g_{N-1}(n-2, p)$. Similarly, with probability $1-p$, taking into account the fact that the direction of stabbing is reverse compared with the previous situation, Participant $n$ carries the label $-n-1\equiv N-n-2\pmod{N-1}$ in a new round with $N-1$ participants in such a way that his probability of survival becomes $g_{N-1}(N-n-2, p)$. Putting these two situations together, one obtains that $$g_N(n,p)\;=\; p \cdot g_{N-1}(n-2, p) + (1-p)\cdot  g_{N-1}(N-n-2, p).$$
\end{itemize}
This completes the proof of the proposition.

\end{proof}

\section[Skewness of the Elimination Process ]{Skewness of the Elimination Process and Rate of Convergence to the Position of the Survivor}\label{sec2}

Throughout this section, unless specified otherwise, the parameter $p\in (0,1)$ and the integer $N\ge 3$ are fixed. The goal is to establish successively each of the three points in Theorem~\ref{mainthm}, the notations of which statements are adopted throughout.\\

\noindent To this end, it is more convenient to work with the closed unit interval $[0,1]$ rather than the space $\R/\Z$. The set of continuous functions $\varphi$ over $\R/\Z$ is then identified with the set of continuous functions over $[0,1]$, denoted by $C^0([0,1])$, with periodic boundary conditions. Given such a function $\varphi$, denote for the sake of simplicity of notations the expectation of the random variable $\varphi\left(X_N^{(p)}\right)$ as 

\begin{equation*}
\mathcal{E}_N^{(p)}[\varphi]\; :=\;  \mathbb{E}^{(p)}\!\left[\varphi\left(X_N^{(p)}\right)\right]
\end{equation*}
in such a way that 
\begin{equation}\label{functJ}
\mathcal{E}_N^{(p)}[\varphi]\; =\; \sum_{n=0}^{N-1}\varphi\left(\frac{n}{N}\right)\cdot g_N(n,p).
\end{equation}
Clearly,
\begin{equation}\label{normjnp}
\left|\mathcal{E}_N^{(p)}[\varphi]\right|\;\le\; \left\|\varphi\right\|_{\infty},
\end{equation}
where $\left\|\varphi\right\|_{\infty}$ denotes the sup--norm of a given bounded map $\varphi$ over its domain of definition.



\subsection{Limiting Behavior of the Random Process when $p\in \left(0,1\right)$}

The first point in Theorem~\ref{mainthm} amounts to claiming that the distributional limit $X^{(p)}$ of the sequence $\left(X^{(p)}_N\right)_{N\ge 3}$ exists and takes value in the set $\left\{0, 1/2\right\}$. This is established in this section as a consequence of a succession of lemmata, each relying on the preceeding ones. The following quantity plays an important \emph{rôle} in the proof~:
\begin{equation}\label{defeta}
\eta_N(p)\;=\; \max_{-2\le n\le 2} g_N(n,p).
\end{equation}


\begin{lem}[Approximate recursion relation for the expectation under regula\-rity assumptions]\label{lem1}
Let $N\ge 4$ be an integer and let $\varphi$ be a twice continuously differentiable function over $[0,1]$. Then,
{\allowdisplaybreaks
\begin{align*}
\mathcal{E}_N^{(p)}[\varphi]\;&=\; \sum_{n=0}^{N-2}g_{N-1}(n,p)\cdot \left[p\cdot\varphi\left(\frac{n}{N-1}\right)+\left(1-p\right)\cdot\varphi\left(1-\frac{n}{N-1}\right)\right.\\
& \left. +\left(\frac{2}{N}-\frac{n}{N(N-1)}\right)\cdot \left(p\cdot\varphi'\left(\frac{n}{N-1}\right)-(1-p)\cdot \varphi'\left(1-\frac{n}{N-1}\right)\right)\right]\\
&+O_{\varphi}\left(\frac{1}{N^{2}}\right)+O\left(\eta_N(p)\cdot \max\left\{\left|\varphi(0)\right|, \left|\varphi(1)\right|, \frac{\left|\varphi'(0)\right|}{N}, \frac{\left|\varphi'(1)\right|}{N}\right\}\right).
\end{align*}
}
\end{lem}

\noindent The subscript in the notation $O_{\varphi}(\,\cdot\,)$ means that the implicit constant in the big-O notation depends on $\varphi$ (whereas it is absolute otherwise).

\begin{proof}
Applying  a Taylor expansion yields that 
{\allowdisplaybreaks
\begin{align*}
\mathcal{E}_N^{(p)}[\varphi]\;&=\; \varphi(0)\cdot g_N(0,p)+\varphi\left(\frac{1}{N}\right)\cdot g_N(1,p)+\varphi\left(\frac{N-1}{N}\right)\cdot g_N(-1, p)\\
& \quad + \sum_{n=2}^{N-2}\varphi\left(\frac{n}{N}\right)\cdot g_N(n,p)\\
&\underset{\eqref{defeta}}{=} \sum_{n=2}^{N-2}\varphi\left(\frac{n}{N}\right)\cdot g_N(n,p) + O\left(\eta_N(p)\cdot \max\left\{\left|\varphi(0)\right|, \left|\varphi(1)\right|, \frac{\left|\varphi'(0)\right|}{N}, \frac{\left|\varphi'(1)\right|}{N}\right\}\right)\\
& \quad + O_{\varphi}\left(\frac{1}{N^{2}}\right).
\end{align*}
}
\noindent From the recursion relations established in Proposition~\ref{recurrel}, this last sum can be expanded as
{\allowdisplaybreaks
\begin{align*}
\sum_{n=2}^{N-2}\varphi\left(\frac{n}{N}\right)\cdot g_N(n,p) \;&=\; p\cdot \sum_{n=2}^{N-2}\varphi\left(\frac{n}{N}\right)\cdot g_{N-1}\left(n-2, p\right)\\
&\qquad +(1-p)\cdot \sum_{n=2}^{N-2}\varphi\left(\frac{n}{N}\right)\cdot g_{N-1}(N-n-2, p)\\
&=\; p\cdot \sum_{n=0}^{N-4}\varphi\left(\frac{n+2}{N}\right)\cdot g_{N-1}\left(n, p\right)\\
&\qquad +(1-p)\cdot \sum_{n=0}^{N-4}\varphi\left(\frac{N-n-2}{N}\right)\cdot g_{N-1}(n, p)\\
&\underset{\eqref{defeta}}{=} \sum_{n=0}^{N-2}\left(p\cdot \varphi\left(\frac{n+2}{N}\right)\cdot g_{N-1}\left(n, p\right)\right.\\
&\qquad \left.+ (1-p)\cdot \varphi\left(1-\frac{n+2}{N}\right)\cdot g_{N-1}(n, p) \right)+O_{\varphi}\left(\frac{1}{N^{2}}\right)\\
&\qquad + O\left(\eta_N(p)\cdot \max\left\{\left|\varphi(0)\right|, \left|\varphi(1)\right|, \frac{\left|\varphi'(0)\right|}{N}, \frac{\left|\varphi'(1)\right|}{N}\right\}\right).
\end{align*}
}
From the decomposition $$\frac{n+2}{N}\;=\; \frac{n}{N-1}+\left(\frac{2}{N}-\frac{n}{N(N-1)}\right),$$ another application of a Taylor expansion yields the relations $$\varphi\left(\frac{n+2}{N}\right)\;=\; \varphi\left(\frac{n}{N-1}\right)+\left(\frac{2}{N}-\frac{n}{N(N-1)}\right)\cdot \varphi'\left(\frac{n}{N-1}\right)+O_{\varphi}\left(\frac{1}{N^2}\right)$$ and $$\varphi\left(1-\frac{n+2}{N}\right)\;=\; \varphi\left(1-\frac{n}{N-1}\right)-\left(\frac{2}{N}-\frac{n}{N(N-1)}\right)\cdot \varphi'\left(1-\frac{n}{N-1}\right)+O_{\varphi}\left(\frac{1}{N^2}\right).$$
Collecting together all the components of the calculation then gives 
\begin{align*}
\mathcal{E}_N^{(p)}[\varphi]\;&=\; \sum_{n=0}^{N-2}g_{N-1}(n,p)\cdot \left[p\cdot \varphi\left(\frac{n}{N-1}+(1-p)\cdot \varphi\left(1-\frac{n}{N-1}\right)\right) \right.\\
&\qquad \left. +\left(\frac{2}{N}-\frac{n}{N(N-1)}\right)\cdot\left(p\cdot \varphi'\left(\frac{n}{N-1}\right)-(1-p)\cdot\varphi'\left(1-\frac{n}{N-1}\right)\right) \right]\\
&\qquad +O_{\varphi}\left(\frac{1}{N^2}\right)+ O\left(\eta_N(p)\cdot \max\left\{\left|\varphi(0)\right|, \left|\varphi(1)\right|, \frac{\left|\varphi'(0)\right|}{N}, \frac{\left|\varphi'(1)\right|}{N}\right\}\right),
\end{align*}
which concludes the proof.
\end{proof}

\noindent In what follows, a function $\varphi~: [0,1]\rightarrow\R$ is referred to as an \emph{odd function about $1/2$} if $\varphi(x)=-\varphi(1-x)$ for all $x\in [0,1]$. Similarly, it is  \emph{even  about $1/2$} if $\varphi(x)=\varphi(1-x)$ for all $x\in [0,1]$.

\begin{lem}[Decay of the expectation under assumptions of periodicity, regularity and oddness]\label{lem2}
Let $N\ge 4$ be an integer and let $\varphi$ be a twice continuously differentiable function over $[0,1]$. Assume that $\varphi$ is odd about 1/2 and that it meets the boundary condition $\varphi(0)=\varphi(1)$. Then, there exists a constant $C(\varphi, p)>0$ depending only on $\varphi$ and $p$ such that $$\mathcal{E}^{(p)}_N[\varphi]\;\le\; \frac{C(\varphi, p)}{N}\cdotp$$
\end{lem}

\begin{proof}
Under the assumptions of the statement (which imply in particular that $\varphi(0)=\varphi(1)=0$), Lemma~\ref{lem1} guarantees the existence of a constant $C'(\varphi, p)>0$ such that for all integers $N> 3$, 
\begin{equation}\label{lemtechle2}
\left|\mathcal{E}_N^{(p)}[\varphi]\right|\;\le\; \left|2p-1\right|\cdot \left|\mathcal{E}_{N-1}^{(p)}[\varphi]\right|+\frac{C'(\varphi, p)}{N}\cdotp
\end{equation} 
By induction, this implies that for all $k\in\left\llbracket 1, N-3\right\rrbracket$, $$\left|\mathcal{E}_N^{(p)}[\varphi]\right|\;\le\; \left|2p-1\right|^k\cdot \left|\mathcal{E}_{N-k}^{(p)}[\varphi]\right|+C'(\varphi, p)\cdot\sum_{l=0}^{k-1}\frac{\left|2p-1\right|^l}{N-l}\cdotp$$Since $|2p-1|<1$, the series $\sum_{l\ge 0}\left|2p-1\right|^l$ converges in such a way that for all $k\in\left\llbracket 1, N-3\right\rrbracket$, $$\left|\mathcal{E}_N^{(p)}[\varphi]\right|\;\le\; \left|2p-1\right|^k\cdot \left\|\varphi\right\|_{\infty}+\frac{C'(\varphi, p)}{\left(N-k+1\right)\cdot \left(1-|2p-1|\right)}\cdotp$$Optimising the choice of  $k$ yields to choose it as the unique integer in the  interval $(N/2-1, N/2]$. Then, $$\left|\mathcal{E}_N^{(p)}[\varphi]\right|\;\le\; \left|2p-1\right|^{\frac{N}{2}-1}\cdot \left\|\varphi\right\|_{\infty}+\frac{C'(\varphi, p)}{\left(N/2+1\right)\cdot \left(1-|2p-1|\right)},$$where the right--hand side decays as a $O(1/N)$ when $N$ tends to infinity. This is readily seen to imply the lemma and thus to conclude the proof.
\end{proof}

\noindent The case that the map $\varphi$ is even about $1/2$ requires more regularity assumptions to analyse the limiting behavior of the real sequence $\left(\mathcal{E}_N^{(p)}[\varphi]\right)_{N\ge 3}$. This is achieved in three steps.

\begin{lem}[Telescopic asymptotic decomposition of the expectation under assumptions of periodicity of the derivative, regularity and evenness]\label{lem3}
Let $N\ge 4$ be an integer and let $\varphi$ be a three times continuously differentiable function over $[0,1]$. Assume that  $\varphi$ is even about 1/2 and that it meets the boundary condition $\varphi'(0)=\varphi'(1)=0$. Then, 
\begin{equation}\label{telescopic1}
\mathcal{E}_N^{(p)}[\varphi]-\mathcal{E}_{N-1}^{(p)}[\varphi]\;=\; \frac{1}{N}\cdotp \mathcal{E}_{N-1}^{(p)}\left[\frac{1}{2}\cdotp \varphi'-\varphi^*\right]+O_{\varphi,p}\left(\frac{1}{N^{2}}\right),
\end{equation} 
where the function $\varphi^*$ is defined as $\varphi^*\; :\; x\mapsto x\cdot \varphi'(x).$
\end{lem}

\noindent As before, the subscripts in the big O notation has the meaning that the implicit constant depends at most on $\varphi$ and $p$.

\begin{proof}
Equation~\eqref{telescopic1} remains invariant upon translation $\varphi$ by a constant~: it may therefore be assumed without loss of generality that $\varphi(0)=\varphi(1)=0$. Since the evenness of $\varphi$ about $1/2$ implies that its derivative is odd about $1/2$, Lemma~\ref{lem1} yields that 
\begin{align}
\mathcal{E}_N^{(p)}[\varphi]\;=\; \mathcal{E}_{N-1}^{(p)}[\varphi] &+ \frac{1}{N}\cdotp \mathcal{E}_{N-1}^{(p)}\left[2\varphi'-\varphi^*\right]+O_{\varphi}\left(\frac{1}{N^{2}}\right)\nonumber\\
&+O\left(\eta_N(p)\cdot\max\left\{\frac{\left|\varphi'(0)\right|}{N}, \frac{\left|\varphi'(1)\right|}{N}\right\}\right),\label{telescopic}
\end{align} 
where the last term vanishes by assumption. Furthermore, from Lemma~\ref{lem2}, 
\begin{align*}
\frac{1}{N}\cdotp \mathcal{E}_{N-1}^{(p)}\left[2\varphi'-\varphi^*\right]\;&=\;  \frac{1}{N}\cdotp \mathcal{E}_{N-1}^{(p)}\left[\frac{1}{2}\cdot \varphi'-\varphi^*\right]+\frac{3}{2N}\cdot \mathcal{E}_{N-1}^{(p)}\left[\varphi'\right]\\
&=\; \frac{1}{N}\cdotp \mathcal{E}_{N-1}^{(p)}\left[\frac{1}{2}\cdot \varphi'-\varphi^*\right]+O_{\varphi, p}\left(\frac{1}{N^2}\right).
\end{align*}
This is enough to complete the proof.
\end{proof}

\begin{coro}[Convergence of the expectation under assumptions of pe\-ri\-odi\-ci\-ty of the derivative, regularity, monotonicity and evenness]\label{lem4} Keep the assumptions of Lemma~\ref{lem3} and assume furthermore that 
$\varphi$ is monotonic increasing on $[0, 1/2]$ and monotonic decreasing on $[1/2, 1]$. Then, the sequence $\left(\mathcal{E}_N^{(p)}[\varphi]\right)_{N\ge 3}$ converges.
\end{coro}

\begin{proof}
From Lemma~\ref{lem3}, given any integers $L>M>3$, 
\begin{equation}\label{teleeq}
\sum_{N=M+1}^L\frac{1}{N}\cdotp \mathcal{E}_{N-1}^{(p)}\left[\frac{1}{2}\cdotp \varphi'-\varphi^*\right]\;=\; \mathcal{E}_L{(p)}[\varphi]\;-\;\mathcal{E}_{M}^{(p)}[\varphi]\; +\; O_{\varphi,p}\left(\frac{1}{M}\right),
\end{equation}
which implies in particular that the series $$\sum_{N\ge 4}\frac{1}{N}\cdotp \mathcal{E}_{N-1}^{(p)}\left[\frac{1}{2}\cdotp \varphi'-\varphi^*\right]$$ is bounded. Since by assumption, $$\left(\frac{1}{2}\cdotp \varphi'-\varphi^*\right)(x)\;=\; \left(\frac{1}{2}-x\right)\cdot \varphi'(x)\ge 0$$for any $x\in [0,1]$, its general term is positive. As a consequence, it converges, hence is Cauchy. Fix then $\varepsilon>0$ and an integer $N(\varepsilon)$ such that for all $L>M\ge N(\varepsilon)$, $$0\;\le\; \sum_{N=M+1}^L\frac{1}{N}\cdotp \mathcal{E}_{N-1}^{(p)}\left[\frac{1}{2}\cdotp \varphi'-\varphi^*\right]\;<\; \frac{\varepsilon}{2}\cdotp$$In view of relation~\eqref{teleeq}, even if it means increasing the value of $N(\varepsilon)$ to absorb the error term, this implies that for all $L>M\ge N(\varepsilon)$, $$\left|\mathcal{E}_L^{(p)}[\varphi]-\mathcal{E}_{M}^{(p)}[\varphi]\right|\;<\; \varepsilon.$$The sequence $\left(\mathcal{E}_N^{(p)}[\varphi]\right)_{N\ge 3}$, being Cauchy, is thus convergent.
\end{proof}

\noindent The same conclusion as in the previous corollary can be obtained  without imposing  boundary periodic conditions on the derivative~:

\begin{coro}[Convergence of the expectation under assumptions of regularity, monotonicity and evenness]\label{lem5}
Let $\varphi$ be a three times continuously differentiable function which is even about $1/2$. Assume that  $\varphi$ is monotonic increasing on $[0, 1/2]$ and monotonic decreasing on $[1/2, 1]$. Then, the sequence $\left(\mathcal{E}_N^{(p)}[\varphi]\right)_{N\ge 3}$ converges.

\end{coro}

\begin{proof}
In view of the hypothesis of evenness, assume without loss of generality that $\varphi(0)=\varphi(1)=0$. By uniform continuity, there exists a function $\epsilon~: \eta>0\mapsto \epsilon(\eta)>0$ tending to zero at the origin such that for any $x,y\in [0,1]$, it holds that $\left|\varphi(x)-\varphi(y)\right|<\epsilon(\eta)$ whenever $\left|x-y\right|<\eta$.\\

\noindent Let then $\left(\zeta_k\right)_{k\ge 3}$ be any sequence of three times continuously differentible maps over $[0,1]$ such that for all $k\ge 3$, 
\begin{itemize}
\item $\zeta_k(0)=\zeta_k(1)=0$ and $\zeta'_k(0)=\zeta'_k(1)=0$;
\item $\zeta_k$ is monotonic increasing on $[0,1/2]$ and monotonic decreasing on $[1/2,1]$ ;
\item $\zeta_k$ is even about $1/2$ and $\zeta_k(x)=1$ for any $x\in \left[1/k, 1-1/k\right]$.
\end{itemize}
From Corollary~\ref{lem4}, the sequence $\left(\mathcal{E}_N^{(p)}[\zeta_k\varphi]\right)_{N\ge 3}$ is convergent, say to a real $\rho_k$, which is bounded in absolute value by $\left\|\varphi\right\|_{\infty}$ (since  so is  each term $\mathcal{E}_N^{(p)}[\zeta_k\varphi]$).\\

\noindent By the construction of the map $\zeta_k$ for any given $k\ge 3$ and by the definition of the expectation $\mathcal{E}_N^{(p)}$ (see~\eqref{functJ}), it then follows that $$\left|\mathcal{E}_N^{(p)}[\varphi]- \mathcal{E}_N^{(p)}[\zeta_k\varphi]\right|\;\le\;\left\|\varphi- \zeta_k\varphi\right\|_{\infty}\;<\; \epsilon\left(\frac{1}{k}\right),$$thence the inequalities
\begin{equation*}
\rho_k-\epsilon\left(\frac{1}{k}\right)\;\le\; \liminf_{N\rightarrow\infty}\mathcal{E}_N^{(p)}[\varphi]\;\le\; \limsup_{N\rightarrow\infty}\mathcal{E}_N^{(p)}[\varphi]\;\le\; \rho_k+\epsilon\left(\frac{1}{k}\right).
\end{equation*}
In particular, 
\begin{equation*}
0\;\le\;\limsup_{N\rightarrow\infty}\mathcal{E}_N^{(p)}[\varphi]- \liminf_{N\rightarrow\infty}\mathcal{E}_N^{(p)}[\varphi]\;\le\; 2\cdot\epsilon\!\left(\frac{1}{k}\right).
\end{equation*}
Upon letting $k$ tend to infinity, this shows that the sequence $\left(\mathcal{E}_N^{(p)}[\varphi]\right)_{N\ge 3}$  converges and thus completes the proof.
\end{proof}

\noindent The convergence of the expectation~\eqref{functJ} evaluated at odd (Lemma~\ref{lem2}) and even (Corollary~\ref{lem5}) functions about $1/2$ under regularity and monoticity assumptions can be ge\-ne\-ra\-lised to the case of any continuous function by a density argument. This is the content of the (proof of the) next statement.

\begin{lem}[Convergence in law of the random variables.]\label{lem6}
In the notations of Theorem~\ref{mainthm}, the sequence of random variables  $\left(X_N^{(p)}\right)_{N\ge 3}$ converges in law to a random variable taking values in $\R/\Z$.
\end{lem}

\begin{proof}
It suffices to prove that for any function $\varphi\in C^0([0,1])$ such that $\varphi(0)=\varphi(1)$ (recall that $\varphi$ is then identified with a continuous function on the unit circle $\R/\Z$), the sequence $\left(\mathcal{E}_N^{(p)}[\varphi]\right)_{N\ge 3}$ converges. Indeed, under this assumption, since the unit circle is compact, the Prohorov Theorem (see~\cite[Theorem~5.1]{bili}) implies that the sequence (of the distributions) of the random variables $\left(X_N^{(p)}\right)_{N\ge 3}$ is relatively compact~: any subsequence contains a convergent one (in law). If the limits of all such subsequences are identically distributed, then the lemma holds. Otherwise, let $\Xi_1$ and $\Xi_2$ be partial limits with nonequal laws. From~\cite[Theorem~1.2]{bili} for instance, there exists a continuous function $f$ over $\R/\Z$ such that  $\mathbb{E}^{(p)}\!\left[f\left(\Xi_1\right)\right]\neq \mathbb{E}^{(p)}\!\left[f\left(\Xi_2\right)\right]$. This nevertheless contradicts the convergence  of the sequence $\left(\mathbb{E}^{(p)}\!\left[f\left(X_N^{(p)}\right)\right]\right)_{N\ge 3}$.


\noindent To establish the lemma, it thus remains to prove that the sequence $\left(\mathcal{E}_N^{(p)}[\varphi]\right)_{N\ge 3}$ indeed converges. To this end, define the auxiliary functions $$\varphi_O~: x\mapsto\frac{\varphi(x)-\varphi(1-x)}{2}\qquad \textrm{and}\qquad \varphi_E~: x\mapsto\frac{\varphi(x)+\varphi(1-x)}{2}\cdotp$$They respectively  represent the odd and even parts of the function $\varphi$ in the sense that they are respectively odd and even about $1/2$ and that $\varphi$ decomposes as $\varphi= \varphi_O+\varphi_E$. From Lemma~\ref{lem3}, it holds that $$\mathcal{E}_N^{(p)}[\varphi]\;=\; \mathcal{E}_N^{(p)}[\varphi_O]+\mathcal{E}_N^{(p)}[\varphi_E]\;=\; \mathcal{E}_N^{(p)}[\varphi_E]+O_{\varphi, p}\left(\frac{1}{N}\right)$$in such a way that it suffices to prove the convergence of the sequence of even parts $\left(\mathcal{E}_N^{(p)}[\varphi_E]\right)_{N\ge 3}$.\\

\noindent To see this, given an integer $k\ge 0$, consider the polynomial map 
\begin{equation}\label{defpolymap}
\varphi_k~: x\in [0,1]\;\mapsto\; \left(\frac{1}{2}-x\right)^k.
\end{equation}
From the Weierstrass approximation theorem,  given any $\varepsilon>0$, there exists a function $\psi_{\varepsilon}$ lying in the real span of the family $\left\{\varphi_k\right\}_{k\ge 0}$ such that $$\left\|\varphi-\psi_{\varepsilon}\right\|_{\infty}\;<\;\frac{\varepsilon}{2}\cdotp$$By the definition of $\varphi_E$, this  implies that 
\begin{equation}\label{inesupeven}
\left\|\varphi_E-\psi_{\varepsilon, E}\right\|_{\infty}\;<\; \varepsilon,
\end{equation}
where $$\psi_{\varepsilon, E}~: x\;\mapsto\; \frac{\psi_{\varepsilon}(x)+\psi_{\varepsilon}(1-x)}{2}$$is the even part (about $1/2$) of $\psi_{\varepsilon}$ and lies as such in the real span of the family of even functions (about $1/2$) $\left\{\varphi_{2k}\right\}_{k\ge 0}$. Corollary~\ref{lem5} is then easily seen to imply the existence of the limit $$\sigma(\varepsilon)\;=\; \lim_{N\rightarrow\infty}\, \mathcal{E}_N^{(p)}[\psi_{\varepsilon, E}].$$Since the set of limit values $\left\{\sigma(\varepsilon)\right\}_{\varepsilon\in (0,1)}$ is clearly bounded, upon setting $\varepsilon=1/k$ and even if it means considering a subsequence, it may be assumed that the sequence $\left(\sigma(1/k)\right)_{k\ge 0}$ converges to a limit $\sigma\ge 0$. Then, 
\begin{align*}
\left|\mathcal{E}_N^{(p)}[\varphi_E]-\sigma\right|\;\le\; \left|\mathcal{E}_N^{(p)}[\varphi_E]-\mathcal{E}_N^{(p)}[\psi_{1/k, E}]\right|+ \left|\mathcal{E}_N^{(p)}[\psi_{1/k, E}]-\sigma(1/k)\right|+\left|\sigma(1/k)-\sigma\right|.
\end{align*}
Since, from inequality~\eqref{normjnp}, the expectation $\mathcal{E}_N^{(p)}$ has a norm bounded by 1, it follows from~\eqref{inesupeven} that the first term on the right-hand can be made arbitrarily small for $k$ large enough. So can the other two terms by the definition of a limit. The right--hand side thus becomes less than any arbitrarily fixed quantity provided that $k$ is large enough~: this shows that the sequence $\left(\mathcal{E}_N^{(p)}[\varphi_E]\right)_{N\ge 3}$ converges to the real $\sigma$.
\end{proof}

\noindent In order to establish  the first point in Theorem~\ref{mainthm}, it remains to to show that the distributional limit $X^{(p)}$ of the sequence $\left(X_N^{(p)}\right)_{N\ge 3}$,  the existence of which is guaranteed by the above Lemma~\ref{lem6}, takes values in the set $\left\{0, 1/2\right\}$.

\begin{proof}[Completion of the proof of Claim (1) in Theorem~\ref{mainthm}]
Upon setting $$\varphi~: x\in [0,1]\mapsto -\cos(2\pi x)\qquad \textrm{and}\qquad\psi~: x\in [0,1]\mapsto 2\pi\cdot \left(\frac{1}{2}-x\right)\cdot \sin(2\pi x),$$ Lemma~\ref{lem3}  yields that for all integers $L>M\ge 3$,  $$\mathcal{E}_L^{(p)}[\varphi]- \mathcal{E}_M^{(p)}[\varphi]\;=\; \sum_{M<N<L}\frac{1}{N}\cdot \mathcal{E}_{N-1}^{(p)}[\psi]\; +\; O_{p}\left(\frac{1}{M}\right)\cdot$$Since, from Lemma~\ref{lem6}, the sequence $\left(\mathcal{E}_N^{(p)}[\varphi]\right)_{N\ge 3}$ converges, hence is Cauchy, the above identity is easily seen to imply that the sequence of partial sums of the series $$\sum_{N\ge 3}\frac{1}{N}\cdot \mathcal{E}_{N-1}^{(p)}[\psi]$$is also Cauchy, hence converges. \sloppy Given that, from Lemma~\ref{lem6} once again, the sequence $\left(\mathcal{E}_N^{(p)}[\psi]\right)_{N\ge 3}$ is also convergent, one thus deduces that $$\lim_{N\rightarrow\infty}\;\mathcal{E}_N^{(p)}[\psi]\;=\; \mathbb{E}\!\left[\psi\left(X^{(p)}\right)\right]\;=\;0.$$


\noindent Since $\psi$ in nonnegative, this implies that $\psi\left(X^{(p)}\right)=0$ almost surely. As a consequence, $X^{(p)}$ takes values in the null set of the map $\psi$, namely $\left\{0, 1/2\right\}$. This concludes the proof of  Claim (1) in Theorem~\ref{mainthm}.
\end{proof}

\subsection{Existence of the Limit Position in the Middle Range $p\in \left(1/3,2/3\right)$}

Throughout this section, the probability parameter $p$ is assumed to lie in the open interval $\left(1/3,2/3\right)$. Under this assumption, the convergence in law of the sequence of random variables $\left(X_N^{(p)}\right)_{N\ge 3}$ to a Bernoulli random variable $X^{(p)}$ taking values in the set $\left\{0,1/2\right\}$  can be made more precise~: the random variable $X^{(p)}$  takes the almost sure value  $1/2$. This is the content of Point 2 in Theorem~\ref{mainthm} which is established hereafter. The key new ingredient allowing one to gather more information on the limit distribution is the following result implying that an exponential decay of the probability of survival occurs near the starting position.

\begin{prop}[Exponential bounds for the probabilities of survival]\label{keypropopmidle}
Assume that $p\in (1/3, 2/3)$. Then, there exist constants $K>0$ and $\beta, \gamma>1$, all depending only on $p$, such that for all integers $N\ge 3$ and $n\in \Z$, $$g_N(n,p)\;\le\; K\cdot \frac{\beta^{\left\langle n\right\rangle_N}}{\gamma^N}\cdotp$$Here, $\left\langle n\right\rangle_N=\textrm{dist}\left(n , N\Z\right)$.
\end{prop}

\begin{proof}
The claimed result is established by an induction making it possible to provide effective values for the various parameters. \\

\noindent First, choose $K>0$ large enough so that the conclusion of the proposition holds for the finitely many probabilities $g_N(n,p)$ determined by the finitely many values of $N\le 7$ and $n \pmod{N}$. Let then $N\ge 8$ be such that this conclusion holds for all integers $L\in\left\llbracket 3, N-1\right\rrbracket$ and $n\pmod{L}$. Consider the following distinction of cases relying on the recursion relations for the probabilities of survival stated in Proposition~\ref{recurrel}~:

\begin{itemize}
\item[$\bullet$] when $n\equiv 0\pmod{N}$, \begin{align*}g_N(0,p)\;&=\; g_{N-1}(-1, p)\;=\; p\cdot g_{N-2}(-2, p)\\
&\le\;  K\cdot \frac{\beta^2}{\gamma^{N-2}}\;\le\; K\cdot \frac{1}{\gamma^N},
\end{align*}
where the last inequality is valid for any choice of $\beta, \gamma>1$ such that $p(\gamma\beta)^2\le 1$;

\item[$\bullet$] when $n\equiv 1\pmod{N}$, \begin{align*}g_N(1,p)\;&=\;(1-p)\cdot g_{N-1}(-2, p)\\
&\le\; (1-p)\cdot K\cdot \frac{\beta^2}{\gamma^{N-1}}\;=\; K\cdot \frac{\beta}{\gamma^N}\cdot \left((1-p)\gamma\beta\right)\;\le\; K\cdot \frac{1}{\gamma^N},\end{align*}
\sloppy where the last inequality is valid for any choice of $\beta, \gamma>1$ such that $(1-p)\gamma\beta\le 1$;

\item[$\bullet$] when $n\equiv -1\pmod{N}$, 
\begin{align*}g_N(1,p)\;&=\;p\cdot g_{N-1}(-2, p)\\
&\le\; p\cdot K\cdot \frac{\beta^2}{\gamma^{N-1}}\;=\; K\cdot \frac{\beta}{\gamma^N}\cdot \left(p\gamma\beta\right)\;\le\; K\cdot \frac{1}{\gamma^N},
\end{align*}
where the last inequality is valid for any choice of $\beta, \gamma>1$ such that $p\gamma\beta\le 1$;

\item[$\bullet$] \sloppy when $n\in \left\llbracket 2,  (N-3)/2\right\rrbracket\pmod{N}$ (then, $\left\langle n-2\right\rangle_{N-1}=n-2$ and $\left\langle N-n-2\right\rangle_{N-1}=n+1$),
\begin{align*} 
g_N(n,p)\;&=\;p\cdot g_{N-1}(n-2, p)+(1-p)\cdot g_{N-1}(N-n-2, p)\\
&\le\; p\cdot K\cdot \frac{\beta^{n-2}}{\gamma^{N-1}}+(1-p)\cdot K\cdot \frac{\beta^{n+1}}{\gamma^{N-1}}\;=\; K\cdot \frac{\beta^n}{\gamma^N}\cdot \left(p\cdot\frac{\gamma}{\beta^2}+(1-p)\cdot \gamma \beta\right)\\
&\le\; K\cdot \frac{\beta^{\left\langle n \right\rangle_{N}}}{\gamma^N},
\end{align*}
where the last inequality is valid for any choice of $\beta, \gamma>1$ such that $\gamma\cdot (p/\beta^2+(1-p)\beta)\le 1$, which can be realised under the assumption that $p>1/3$;

\item[$\bullet$] when $n\in \left\llbracket (N+3)/2, N-2\right\rrbracket\pmod{N}$ (then, $\left\langle n-2\right\rangle_{N-1}=N-n+1$ and $\left\langle N-n-2\right\rangle_{N-1}=N-n-2$),
\begin{align*} 
g_N(n,p)\;&=\;p\cdot g_{N-1}(n-2, p)+(1-p)\cdot g_{N-1}(N-n-2, p)\\
&\le\; p\cdot K\cdot \frac{\beta^{N-n+1}}{\gamma^{N-1}}+(1-p)\cdot K\cdot \frac{\beta^{N-n-2}}{\gamma^{N-1}}\\
&=\; K\cdot \frac{\beta^{N-n}}{\gamma^N}\cdot \left(p\gamma\beta+(1-p)\cdot \frac{\gamma}{ \beta^2}\right)\;\le\; K\cdot \frac{\beta^{\left\langle n \right\rangle_{N}}}{\gamma^N},
\end{align*}
where the last inequality is valid for any choice of $\beta, \gamma>1$ such that $\gamma\cdot (p\beta+(1-p)/\beta^2)\le 1$, which can be realised under the assumption that $p<2/3$;

\item[$\bullet$] when $n\in \left\llbracket (N-3)/2, (N+3)/2\right\rrbracket\pmod{N}$ (then, $\left\langle n-2\right\rangle_{N-1}=n-2$ and $\left\langle N-n-2\right\rangle_{N-1}=N-n-2$),
\begin{align*} 
g_N(n,p)\;&=\;p\cdot g_{N-1}(n-2, p)+(1-p)\cdot g_{N-1}(N-n-2, p)\\
&\le\; p\cdot K\cdot \frac{\beta^{n-2}}{\gamma^{N-1}}+(1-p)\cdot K\cdot \frac{\beta^{N-n-2}}{\gamma^{N-1}}\\
&=
\begin{cases}
\; K\cdot \frac{\beta^{n}}{\gamma^N}\cdot \left(p\frac{\gamma}{\beta^2}+(1-p)\cdot \beta^{N-2n}\frac{\gamma}{ \beta^2}\right) &\textrm{ if } (N-3)/2\le n\le N/2;\\
\; K\cdot \frac{\beta^{N-n}}{\gamma^N}\cdot \left(p\cdot \beta^{2n-N}\frac{\gamma}{\beta^2}+(1-p)\frac{\gamma}{ \beta^2}\right)&\textrm{ if } N/2\le n\le (N+3)/2;
\end{cases}\\
&\le 
\begin{cases}
\; K\cdot \frac{\beta^{n}}{\gamma^N}\cdot \left(p\frac{\gamma}{\beta^2}+(1-p)\cdot \beta \gamma\right) &\textrm{ if } (N-3)/2\le n\le N/2;\\
\; K\cdot \frac{\beta^{N-n}}{\gamma^N}\cdot \left(p\cdot \beta \gamma+(1-p)\frac{\gamma}{ \beta^2}\right)&\textrm{ if } N/2\le n\le (N+3)/2;
\end{cases}\\
& \le\; K\cdot \frac{\beta^{\left\langle n \right\rangle_{N}}}{\gamma^N},
\end{align*}
where the last inequality is valid for any choice of $\beta, \gamma>1$ such that $\gamma\cdot \max\left\{p\beta+(1-p)/\beta^2, p/\beta^2+(1-p)\beta\right\}\le 1$, which can be realised under the assumption that $1/3<p<2/3$.
\end{itemize}
Since, under the assumption that $1/3<p<2/3$, all four inequalities $p\beta^2\gamma^2\le 1$, $(1-p)\beta\gamma\le 1$, $\gamma\cdot (p\beta+(1-p)/\beta^2)\le 1$ and $\gamma\cdot ((1-p)\beta+p/\beta^2)\le 1$ can be met simultaneously, the proof of the proposition is complete.
\end{proof}

\noindent The following two results are derived from the above statement. They should be seen as variants of~Lemmata~\ref{lem2} and~\ref{lem3} proved in the previous section in this sense~: they enable one to retrieve the respective conclusions of those lemmata upon replacing the boundary condition assumptions they include with the assumption that the parameter $p$ lies in the middle interval $(1/3, 2/3)$.

\begin{lem}[Decay of the expectation under assumptions of regularity, oddness and localisation of the parameter $p$]\label{lem2bis}
Let $N\ge 4$ be an integer and let $\varphi$ be a twice continuously differentiable function over $[0,1]$. Let furthermore $p$ be a parameter lying in the interval $\left(1/3, 2/3\right)$. Assume that the function $\varphi$ is odd about 1/2. Then, the conclusion of Lemma~\ref{lem2} still holds. 
\end{lem}

\begin{proof}
The exponential decay of the sequence $\left(\eta_N(p)\right)_{N\ge 3}$ defined from~\eqref{defeta}, which is ensured by Proposition~\ref{keypropopmidle}, implies that the inequality~\eqref{lemtechle2} still holds. The proof of Lemma~\ref{lem2} from that point on is therefore still valid.
\end{proof}

\begin{lem}[Telescopic asymptotic of the expectation under assumptions of regularity, evenness and localisation of the parameter $p$]\label{lem3bis}
Let $N\ge 4$ be an integer and let $\varphi$ be a three times continuously differentiable function over $[0,1]$. Let furthermore $p$ be a parameter lying in the interval $\left(1/3, 2/3\right)$. Assume that  $\varphi$ is even about 1/2. Then, the conclusion of Lemma~\ref{lem3} still holds. 
\end{lem}

\begin{proof}
The exponential decay of the sequence $\left(\eta_N(p)\right)_{N\ge 3}$  ensured by Proposition~\ref{keypropopmidle} implies that the last error term in equation~\eqref{telescopic} may be absorbed in the preceeding error term $O_{\varphi, p}\left(N^{-2}\right)$. This reduces the proof to that of Lemma~\ref{lem3}.
\end{proof}


\begin{proof}[Completion of the proof of Point 2 in Theorem~\ref{mainthm}] The goal is to show that the distributional limit $X^{(p)}$ of the sequence of random variables $\left(X_N^{(p)}\right)_{N\ge 3}$ takes the almost sure value $1/2$  under the assumption that $p\in (1/3, 2/3)$. To this end, let $\varphi$ be a thrice continuously differentiable function even about $1/2$ on $[0,1]$. Assume that $\varphi$ is monotonic increasing on $[0, 1/2]$ and monotonic decreasing on $[1/2, 1]$. From Lemma~\ref{lem3bis}, given any integers $L>M\ge 3$, 
\begin{equation}\label{cauchyop}
\mathcal{E}_L^{(p)}[\varphi]-\mathcal{E}_M^{(p)}[\varphi]\;=\; \sum_{M<N\le L}\frac{1}{N}\cdot \mathcal{E}_{N-1}^{(p)}\left[\frac{1}{2}\cdot \varphi'-\varphi^*\right]+O_{\varphi, p}\left(\frac{1}{M}\right).
\end{equation}
This implies that the partial sums of the series $$\sum_{N\ge 3}\frac{1}{N}\cdot \mathcal{E}_{N-1}^{(p)}\left[\frac{1}{2}\cdot \varphi'-\varphi^*\right]$$are bounded. Since the conditions placed onto $\varphi$ imply that the general term of this series is a nonnegative sequence, it converges, hence is Cauchy. From relation~\eqref{cauchyop}, this claim also holds for the sequence $\left(\mathcal{E}_N^{(p)}[\varphi]\right)_{N\ge 3}$, which is therefore also convergent.\\

\noindent In the particular case that $\varphi=-\varphi_{2k} $ for some integer $k\ge 1$ (where the polynomial map $\varphi_{2k} $ is defined in~\eqref{defpolymap}), it holds that $\frac{1}{2}\left(-\varphi'_{2k}\right)-\left(-\varphi_{2k}\right)^*\;=\;2k\cdot\varphi_{2k}$ in such a way that one obtains the respective convergence of the series and of the sequence $$\sum_{N\ge 3}\frac{1}{N}\cdot \mathcal{E}_{N-1}^{(p)}[\varphi_{2k}]\qquad \textrm{and}\qquad \left(\mathcal{E}_N^{(p)}[\varphi_{2k}]\right)_{N\ge 3}.$$As a conquence, for any $k\ge 1$, 
\begin{equation}\label{limspaneven}
\lim_{N\rightarrow\infty}\; \mathcal{E}_N^{(p)}[\varphi_{2k}]\;=\; 0.
\end{equation}

\noindent Consider now the general case where $\varphi\in C^0([0,1])$ meets the boundary condition $\varphi(0)=\varphi(1)$. As in the proof of Lemma~\ref{lem6}, decompose it into odd and even parts about $1/2$, viz.~$\varphi=\varphi_O+\varphi_E$. Then, the odd part $\varphi_O$ is such that $\varphi_O(0)=\varphi_O(1)=0$. From Lemma~\ref{lem2bis}, this implies that $\lim_{N\rightarrow\infty} \mathcal{E}_N^{(p)}[\varphi_O] = 0$. As for the even part, as in the proof of Lemma~\ref{lem6}, it can be uniformly approximated in the real span of the polynomial maps $\left\{\varphi_{2k}\right\}_{k\ge 0}$. The constant term in such an approximation to $\varphi_E$ in the real span of  $\left\{\varphi_{2k}\right\}_{k\ge 0}$ approximates 
$\varphi_E(1/2)=\varphi(1/2)$. In view of the limit relation~\eqref{limspaneven} valid for all $k\ge 1$, one thus quickly retrieves that $$\lim_{N\rightarrow\infty} \; \mathcal{E}_N^{(p)}[\varphi]\;=\;\varphi\left(\frac{1}{2}\right),$$which concludes the proof~\textsuperscript{6}. 

\let\thefootnote\relax\footnotetext{\textsuperscript{6} {As pointed out by the referee, Point 2 in Theorem~\ref{mainthm} can be established in a quicker way without resorting to Lemmata~\ref{lem2bis} and~\ref{lem3bis}. The suggested argument relies on a classical result in probability theory and can be sketched as follows~: fix a small enough $\delta>0$. Then, the probability of the event $\left\{0\le X_N<\delta\right\}\cup \left\{X_N>1-\delta\right\}$ is bounded above by $$\sum_{|n|\le\delta N} g_N(n,p)\; \le\; 2K\sum_{0 \le n\le \delta N} \frac{\beta^n}{\gamma^N}\;\le\; 2KN\frac{\beta^{\delta N}}{\gamma^N}\cdotp$$ The right--hand side tends to zero as $N$ tends to infinity if, and only if, $\delta \log \beta < \log \gamma$. Under this assumption, Theorem 2.1 in \cite{bili} implies that the limit distribution $\xi$ of the sequence of random variables $\left(X_N\right)_{N\ge 3}$ assigns a zero mass to the $\delta$--neighborhood of 0 on the unit circle. As a consequence, the probability of the event $\{\xi=0\}$ vanishes which, from Point 1 in  Theorem~\ref{mainthm} , implies that the probability of the event $\{\xi=1\}$ is 1, as desired.}}
\end{proof}

\subsection{A Central Limit Theorem in the unbiased Case $p=1/2$}

Assume throughout this section that $p=1/2$. For the sake of simplicity of notations, set then 
\begin{align*}
\eta_N\;=\;& \eta_N\left(\frac{1}{2}\right), \qquad \mathbb{E}_N\;=\;\mathbb{E}_N^{(1/2)},\qquad\mathcal{E}_N\left[\varphi\right]\;=\; \mathcal{E}_N^{(1/2)}\left[\varphi\right]\\ 
&X_N\;=\; X_N^{(1/2)} \quad\textrm{and}\quad g_N(n)\;=\; g_N\left(n,\frac{1}{2}\right).
\end{align*}

\noindent In this setup, the probability for the  participant in position $n$ to survive in a round with $N$ participants is the same as the probability for the participant in position $-n \pmod{N}$ to survive in a round with $N$ participants when the directions of elimination, which are chosen equiprobably, are switched. As a consequence, the symmetry property 
\begin{equation}\label{symprounbia}
g_N(n)=g_N(-n)
\end{equation}
is met and the recurrence relations stated in Proposition~\ref{recurrel} simplify to

\begin{equation}\label{symprob}
g_N(n) = \begin{cases}
  g_{N-1}(-1)  &  \text{ if }  n \equiv 0 \pmod{N}  \\
  g_{N-1}(-2)/2 & \text{ if } n \equiv \pm 1 \pmod{N}\\
  \left(g_{N-1}(n-2) + g_{N-1}(n+1)\right)/2  &  \text{ if } n\not\equiv \pm 1, 0\pmod{N}.  
\end{cases}
\end{equation}

\noindent The goal in this section is to establish, in three steps, the Central Limit Theorem stated in Point 3 of Theorem~\ref{mainthm}.

\subsubsection{Decay Rate of the Probabilities of Survival.} 

The key result allowing one to obtain the convergence rate to the constant $1/2$ in the form of a Central Limit Theorem is the following one, which shows that participants away from the mid-position in the circle have a probability of survival decreasing exponentially fast. This statement refines that of Proposition~\ref{recurrel} (which was concerned with the case where $1/3<p<2/3$) and heavily relies on various symmetries met by the elimination process  in the unbiased case $p=1/2$.

\begin{prop}[Exponential decay of the probability of survival away from the mid-point position in the unbiased case]\label{propdecayCLT} There exist a constant $K>0$  such that given any sufficiently small $\varepsilon\in (0, 1)$ and  any real $\alpha\in (1, 1+\varepsilon]$,  it holds that for all integers $N\ge 1$ and $n\in \left\llbracket 0, N-1\right\rrbracket$, 
\begin{equation}\label{probboununbioa}
g_N(n)\;\le \; K \cdot \alpha^{2(1+\varepsilon)n-N}.
\end{equation}
\end{prop}

\begin{proof}
Fix first $\varepsilon\in (0,1]$. Elementary calculations then show that the inequality in the variable $\alpha\ge 0$ 
\begin{equation}\label{ineqcalcubas} 
\max\left\{\alpha^{2+4(1+\varepsilon)},\; \alpha^{1-4(1+\varepsilon)}+\alpha^{1+2(1+\varepsilon)}\right\}\;\le\; 2
\end{equation} 
is met as soon as
\begin{equation*}\label{boundineqcalc}
\alpha\;\le\;\alpha(\varepsilon), \quad \textrm{where}\quad \alpha(\varepsilon)\;=\;\min\left\{2^{1/10}, 1+4^{2+\varepsilon}\cdot \varepsilon\cdot (1+\varepsilon)^2\right\},
\end{equation*}
and in  particular for any
\begin{equation*}\label{boundineqcalcbis}
1\;<\;\alpha\;\le\;1+\varepsilon
\end{equation*}
provided that $\varepsilon$ is small enough.\\

\noindent This observation enables one to develop a proof of the statement analogous to that of Proposition~\ref{recurrel}. To see this, fix any constant $K\ge 1$ such that the upper bound in~\eqref{probboununbioa} is verified for the finitely many values of the probabilities corresponding to the integers $N\le 4$ for the two limit values $\alpha=1$ and $\alpha=2$, already when $\varepsilon=0$. Fix then $\varepsilon>0$ small enough so that for any $\alpha\in (1, 1+\varepsilon]$,  inequality~\eqref{ineqcalcubas} is satisfied. Given an integer $N\ge 5$  for which the bound~\eqref{probboununbioa} holds for all probabilities up to the range $N-1$, consider the following distinction of cases relying on the recursion relations~\eqref{symprob}. In view of the symmetry property~\eqref{symprounbia}, they imply  inequality~\eqref{probboununbioa} for all admissible values of the integer $n\pmod{N}$~: 

\begin{itemize}
\item[$\bullet$] when $n\equiv 0\pmod{N}$, 
\begin{align*}
g_N(0)&\;=\; g_{N-1}(-1)\;=\; \frac{1}{2}\cdot g_{N-2}(-2)\;\underset{\eqref{symprounbia}}{=}\; \frac{1}{2}\cdot g_{N-2}(2)\\
&\le\; K\cdot \frac{1}{2}\cdot \alpha^{2(1+\varepsilon)\cdot 2- (N-2)}\;=\; K\cdot \alpha^{-N}\cdot \left(\frac{1}{2}\cdot\alpha^{2+4(1+\varepsilon)}\right)\\
&\underset{\eqref{ineqcalcubas}}{\le }\; K\cdot \alpha^{-N}.
\end{align*}

\item[$\bullet$] when $n\equiv 1\pmod{N}$, 
\begin{align*}
g_{N}(1)&\;=\; \frac{1}{2}\cdot g_{N-1}(-2)\;\underset{\eqref{symprounbia}}{=}\; \frac{1}{2}\cdot g_{N-1}(2)\\
&\le\; K\cdot \frac{1}{2}\cdot \alpha^{2(1+\varepsilon)\cdot 2- (N-1)}\;=\; K\cdot \alpha^{2(1+\varepsilon)-N}\cdot \left(\frac{1}{2}\cdot\alpha^{1+2(1+\varepsilon)}\right)\\
&\underset{\eqref{ineqcalcubas}}{\le }\; K\cdot \alpha^{-N}.
\end{align*}

\item[$\bullet$] when $n\in \left\llbracket 2, N/2\right\rrbracket\pmod{N}$, 
\begin{align*}
g_{N}(n)&\;=\; \frac{1}{2}\cdot g_{N-1}(n-2)+ \frac{1}{2}\cdot g_{N-1}(n+1)\\
&\le\; K\cdot \frac{1}{2}\cdot \alpha^{2(1+\varepsilon)\cdot (n-2)- (N-1)}+ K\cdot \frac{1}{2}\cdot \alpha^{2(1+\varepsilon)\cdot (n+1)- (N-1)}\\
&=\; K\cdot \alpha^{2(1+\varepsilon)n-N} \cdot \left(\frac{1}{2}\cdot\alpha^{1-4(1+\varepsilon)}+\frac{1}{2}\cdot\alpha^{1+2(1+\varepsilon)}\right)\\
&\underset{\eqref{ineqcalcubas}}{\le }\; K\cdot \alpha^{2(1+\varepsilon)n-N}.
\end{align*}
\end{itemize}
This completes the proof of the proposition.
\end{proof}

\subsubsection{Some Moment Estimates.}

The Central Limit Theorem relies on various moment estimates for the polynomial maps already introduced in the proof of Lemma~\ref{lem6}, namely, given an integer $k\ge 1$, $$\varphi_k\;:\; x\in [0,1]\;\mapsto\;\left(\frac{1}{2}-x\right)^k.$$

\begin{lem}[Moment Estimates]\label{momentestim}
Let $k\ge 1$. Then, there exists a real constant $\theta_k>0$ such that  $$\mathcal{E}_N[|\varphi_k|]\;\le\; \theta_k\cdot \left(\frac{\ln N}{N}\right)^{k/2}.$$
\end{lem}

\begin{proof}
Given $\varepsilon\in (0,1)$ and $N\ge 1$, define 
\begin{equation*}
\mathcal{E}_{\varepsilon}(N)\;=\; \left\llbracket 0, N \right\rrbracket\backslash \left\llbracket \left(\frac{1}{2}-\varepsilon\right)\cdot N, \left(\frac{1}{2}+\varepsilon\right)\cdot N  \right\rrbracket \quad \textrm{and}\quad \nu_{\varepsilon}(N)\;=\; \sum_{n\in \mathcal{E}_{\varepsilon}(N)} g_N(n). 
\end{equation*}
Then, 
\begin{align}
\mathcal{E}_N[|\varphi_k|]\;&=\; \left(\sum_{\underset{0\le n\le N-1}{n\not\in \mathcal{E}_{\varepsilon}(N)}}+ \sum_{\underset{0\le n\le N-1}{n\in \mathcal{E}_{\varepsilon}(N)}} \right) \left|\frac{1}{2}-\frac{n}{N}\right|^k\cdot g_N(n)\nonumber\\
&\le\; \varepsilon^k \cdot\left(1-\nu_{\varepsilon}(N)\right)+\nu_{\varepsilon}(N).\label{ineqmajunbi}
\end{align}
Furthemore, from Proposition~\ref{propdecayCLT}, there exists a constant $K>0$ such that, provided that $\varepsilon>0$ is chosen small enough, 
\begin{align*}
\nu_{\varepsilon}(N)\;&\underset{\eqref{symprounbia}}{=}\; 2\cdot \left(\sum_{0\le n\le (1/2-\varepsilon)N}g_N(n)\right)\\
&\le \; 2\cdot K\cdot\left( \sum_{0\le n\le (1/2-\varepsilon)N}\left(1+\varepsilon\right)^{2(1+\varepsilon)n-N}\right)\\
&\le \; 8\cdot K\cdot\frac{(1+\varepsilon)^{2(1+\varepsilon)(1/2-\varepsilon)N}}{\left((1+\varepsilon)^{2}-1\right)\cdot (1+\varepsilon)^{N}} \\
&\le \; 8\cdot K\cdot\frac{1}{\varepsilon\cdot (1+\varepsilon)^{\varepsilon N}}\\
&\le \; 8\cdot K\cdot\frac{1}{\varepsilon\cdot \exp\left(\varepsilon^2N/2\right)},
\end{align*}
where the last inequality follows from an easily verified convexity inequality. The last quantity is minimised when 
\begin{equation}\label{choicevareps}
\varepsilon\;= \; \varepsilon_k \;=\; \sqrt{(k+1)\cdot\frac{\ln N}{N}},
\end{equation} 
in which case one infers the existence of a real $\theta_k>0$ such that 
\begin{equation}\label{majnuvare}
\nu_{\varepsilon_k}(N)\;\le\;\theta_k\cdot \frac{1}{N^{k/2}\cdot\sqrt{\ln N}}\cdot \end{equation}
The right--hand side of inequality~\eqref{ineqmajunbi} specialised to the cases when  relations~\eqref{choicevareps} and~\eqref{majnuvare} hold then yields the sought conclusion upon adjusting the value of the real $\theta_k$.
\end{proof}

\noindent Lemma~\ref{momentestim} can be further refined  when restricting to the first and the second moments~:

\begin{lem}[Refined first and second moment estimates]\label{refinedmomentestim} The sum of the second moments satisfies the estimate
\begin{equation*}\label{2ndmoment}
\sum_{N=3}^{L}\mathcal{E}_N[\varphi_2]\;\asymp\; \ln L
\end{equation*}
for all $L$ large enough. As for the first moments, they decay exponentially in the sense that there exists a parameter $A>1$ and a constant $\theta>0$ such that for all $N\ge 3$, 
\begin{equation*}\label{1stmoment}
\left|\mathcal{E}_N[\varphi_1]\right|\;\le\; \frac{\theta}{A^N}\cdotp
\end{equation*}
\end{lem}

\noindent The proof of Lemma~\ref{refinedmomentestim} relies on an auxiliary statement analogous to Lemma~\ref{lem1}. As the argument to establish the former auxiliary statement is very similar to one developed for the latter (upon calling, in this case, on the recursion relations~\eqref{symprob}), it is left to the reader.

\begin{lem}[Telescopic decomposition of the expectation in the unbiased case]\label{telunbia}
Let $\varphi$ be a three-times continuously differentiable function and let $N\ge 4$ be an integer. Then, $$\mathcal{E}_N[\varphi]-\mathcal{E}_{N-1}[\varphi]\;=\;\frac{1}{N}\cdot \mathcal{E}_{N-1}\left[\frac{1}{2}\varphi'-\varphi^*\right]+\frac{C_N}{N^2}+O_{\varphi}\left(\max\left\{\eta_N, \frac{\left\|\varphi^{'''}\right\|_{\infty}}{N^3}\right\}\right),$$where the  ${ }^*$ operator is here again defined by $\varphi^*~: x\mapsto x\cdot \varphi'(x)$ and 
where $$C_N\;=\; \frac{1}{2}\cdot\sum_{n=0}^{N-2}\left(\frac{1}{2}\cdot \left(2-\frac{n}{N-1}\right)^2+\frac{1}{2}\cdot \left(1+\frac{n}{N-1}\right)^2\right)\cdot \varphi''\left(\frac{n}{N-1}\right)\cdot g_{N-1}(n).$$
\end{lem}

\begin{proof}[Proof of Lemma~\ref{refinedmomentestim}]
Applying Lemma~\ref{telunbia} to the polynomial map $\varphi_2$, one obtains that 
\begin{equation}\label{telepolymom2}
\mathcal{E}_N[\varphi_2]\;=\;\mathcal{E}_{N-1}[\varphi_2]-\frac{2}{N}\mathcal{E}_{N-1}[\varphi_2]+\frac{C_N}{N^2}+O\left(\eta_N\right),
\end{equation}
where, taking into account the definition of  the quantity $C_N$ and  the fact that $\varphi_2''$ is the constant function equal to 2, $$C_N\;=\; \mathcal{E}_N[\sigma]\qquad \textrm{with}\qquad \sigma(x)\;=\; \frac{(2-x)^2}{2}+\frac{(1+x)^2}{2}\cdotp$$From Point 2 in Theorem~\ref{mainthm}, one thus infers that 
\begin{equation}\label{limCN}
\lim_{N\rightarrow\infty} C_N\;=\; \sigma\left(\frac{1}{2}\right)\;=\; \frac{9}{4}\cdotp
\end{equation}
Also, by partial summation, given an integer $M\ge 4$,
\begin{align}
\sum_{N=3}^{M}\mathcal{E}_N[\varphi_2]\;=&\; (M+1)\cdot \mathcal{E}_M[\varphi_2]-3\cdot \mathcal{E}_3[\varphi_2]-\sum_{N=4}^{M}N\cdot\left(\mathcal{E}_N[\varphi_2]-\mathcal{E}_{N-1}[\varphi_2]\right)\nonumber \\
&\underset{\eqref{telepolymom2}}{=}\;  (M+1)\cdot \mathcal{E}_M[\varphi_2]-3\cdot  \mathcal{E}_3[\varphi_2]+2\cdot \left(\sum_{N=4}^M \mathcal{E}_{N-1}[\varphi_2]\right)\nonumber \\
&\qquad -\sum_{N=4}^M\left(\frac{C_N}{N}+O\left(N\cdot\eta_N\right)\right).\label{latenb}
\end{align}
Since Proposition~\ref{propdecayCLT} guarantees the existence of constants $K>0$ and $A>1$ such that $\eta_N\le K\cdot A^{-N}$ for all $N\ge 1$, one deduces from the norm inequality~\eqref{normjnp} that 
\begin{equation}\label{asympMlarg}
\sum_{N=3}^{M-1} \mathcal{E}_N[\varphi_2]+(M-2)\cdot \mathcal{E}_M[\varphi_2]\;=\; \sum_{N=4}^{M}\frac{C_N}{N} + O(1)\;\underset{\eqref{limCN}}{\asymp}\; \ln M,
\end{equation}
where the last relation holds for all $M\ge 4$ large enough. Upon summing up this identity, it implies with the help of an elementary manipulation of equation~\eqref{latenb} that for all $L\ge 4$ large enough, 
\begin{align*}
\left(L-2\right)\cdot \sum_{M=3}^{L} \mathcal{E}_M[\varphi_2]\;&=\; \sum_{M=3}^{L}\left(\sum_{N=3}^{M-1}\mathcal{E}_N[\varphi_2]+(M-2)\cdot \mathcal{E}_M[\varphi_2]\right)\\
&=\; \sum_{M=3}^{L}\left(\sum_{N=4}^{M}\frac{C_N}{N}+O(1)\right)\\
&\asymp\; L\cdot\ln L.
\end{align*}
This is easily seen to imply the estimate  for the second moment in the statement of the lemma.\\


\noindent As for the first moment, note that by the symmetry property~\eqref{symprounbia}, the random variables $X_N$ and $1-X_N-\chi_{\left\{X_N=0\right\}}$ (where $\chi_A$ denotes the characteristic function of the event $A$) share a common distribution. Upon taking their expectations, one obtains that $$\mathcal{E}_N[\varphi_1]\;=\; \frac{g_N(0)}{2}\cdotp$$

\noindent  Since from Proposition~\ref{propdecayCLT}, there exist constants $\theta>0$ and $A>1$ such that $g_N(0)\le \theta\cdot A^{-N}$ for all $N\ge 1$, this suffices to conclude the proof.
\end{proof}

\subsubsection{Verifying the Lyapunov Condition.}\label{secvari} 

In order to complete the proof of the Central Limit Theorem, consider a sequence of random variables $\left(X_N\right)_{N\ge 3}$ on the unit circle (identified with the interval $[0,1)$), each drawn according to the discrete law assigning the probability $g_N(n)$ to the event $X_N=n/N$ when $0\le n\le N-1$. Let then 
\begin{equation*}
\mathbb{E}_N(X_N)\;=\;\sum_{n=0}^{N-1}\frac{n}{N}\cdot g_N(n), \qquad \mathbb{V}_N(X_N)\;=\; \sum_{n=0}^{N-1}\left(\frac{n}{N}-\mathbb{E}_N(X_N)\right)^2\cdot g_N(n) 
\end{equation*}
and 
\begin{equation*}
\mathbb{W}_N(X_N)\;=\; \sum_{n=0}^{N-1}\left|\frac{n}{N}-\mathbb{E}_N(X_N)\right|^3\cdot g_N(n). 
\end{equation*}

\noindent When $L\ge 3$ is an integer, set furthermore 
\begin{equation*}
B_L\;=\; \sqrt{\sum_{N=3}^L \mathbb{V}_N(X_N)}.
\end{equation*}

\noindent A classical version of the Central Limit Theorem applicable to the case where random variables behave independently without being necessarily identically distributed is due to Lyapunov~\cite[p.362]{bil}. In the present case, it asserts that if the $X_N$'s are drawn independently from each other, then the convergence in law
\begin{equation}\label{convlyap}
\frac{1}{B_L}\cdot\sum_{N=1}^{L}\left(X_N-\mathbb{E}_N(X_N)\right)\;\overset{\mathcal{L}}{\underset{L\rightarrow\infty}{\longrightarrow}}\; \mathcal{N}(0,1)
\end{equation}  
holds provided that the \emph{Lyapunov condition}
\begin{equation}\label{condlyap}
\lim_{N\rightarrow\infty}\left(\frac{1}{B_L^{2+\kappa}}\cdot \sum_{N=3}^L \mathbb{E}_N\left[\left|X_N-\mathbb{E}_N(X_N)\right|^{2+\kappa}\right]\right)\;=\; 0
\end{equation} 
is met for some $\kappa>0$. The goal in this section is to show that this condition is indeed verified and then to derive Point 3 in Theorem~\ref{mainthm} from~\eqref{convlyap}.

%

\begin{proof}[Proof of Point 3 in Theorem~\ref{mainthm}.]
Note first that
\begin{align*}
\frac{1}{B_L}\cdot \sum_{N=1}^{L}\left(X_N-\frac{1}{2}\right)\;&=\; \frac{1}{B_L}\cdot \sum_{N=1}^{L}\left(X_N- \mathbb{E}_N(X_N)\right)+\frac{1}{B_L}\cdot \sum_{N=1}^{L}\left(\mathbb{E}_N(X_N)-\frac{1}{2}\right)\\
&=\; \frac{1}{B_L}\cdot \sum_{N=1}^{L}\left(X_N- \mathbb{E}_N(X_N)\right)-\frac{1}{B_L}\cdot \sum_{N=1}^{L}\mathcal{E}_N[\varphi_1].
\end{align*}
From this decomposition and from the classical Slutsky Theorem, the convergence in the statement of Point 3 in Theorem~\ref{mainthm} follows upon establishing these two points~:

\begin{itemize}
\item[(a)] the limit relation $$\lim_{L\rightarrow\infty}\;\frac{1}{B_L}\cdot\left( \sum_{N=1}^{L}\mathcal{E}_N[\varphi_1]\right)\;=\; 0$$holds;

\item[(b)] the Lyapunov condition~\eqref{condlyap} is verified when $\kappa=1$  (so that the convergence in law~\eqref{convlyap} is valid).
\end{itemize}

\noindent To this end, note first that Lemma~\ref{refinedmomentestim} implies that the  series $$ \sum_{N\ge 3} \mathcal{E}_N[\varphi_1]$$is absolutely convergent. This also holds for the series $$\sum_{N\ge 3} \mathbb{E}_N\left[\left|X_N-\mathbb{E}_N(X_N)\right|^{3}\right]\;=\; \sum_{N\ge 3}\mathbb{W}_N(X_N).$$Indeed, this claim immediately follows from Lemma~\ref{momentestim} upon noticing that 
\begin{align*}
\mathbb{W}_N(X_N)\;&\le\; \sum_{n=0}^{N-1}\left(\left|\frac{n}{N}-\frac{1}{2}\right|+\left|\frac{1}{2}-\mathbb{E}_N(X_N)\right|\right)^3\cdot g_N(n)\\
&=\; \sum_{n=0}^{N-1}\left(\left|\frac{n}{N}-\frac{1}{2}\right|+\left|\mathcal{E}_N[\varphi_1]\right|\right)^3\cdot g_N(n)\\
&\le\; \mathcal{E}_N[|\varphi_3|]+3\cdot \left(\mathcal{E}_N[|\varphi_2|]\right)^2\cdot \mathcal{E}_N[|\varphi_1|] + 4\cdot \left(\mathcal{E}_N[|\varphi_1|]\right)^3.
\end{align*}
To show that both Points (a) and (b) hold, it thus suffices that the sequence $\left(B_L\right)_{L\ge 1}$ should tend to infinity. This can be established in the quantitative way required by the statement of Theorem~\ref{mainthm}. To see this, note that when $L\ge 3$ is an integer, $$B_L^2\;=\; \sum_{N=3}^{L}\mathbb{V}_N(X_N)\;=\; \sum_{N=3}^{L}\mathcal{E}_N[\varphi_2]+3\cdot\left(\sum_{N=3}^{L}\mathcal{E}_N[\varphi_1]^2\right).$$In this decomposition, from Lemma~\ref{refinedmomentestim}, $$\sum_{N=3}^{L}\mathcal{E}_N[\varphi_2]\;\asymp\; \ln L\qquad \textrm{ and }\qquad \sup_{L\ge 3}\;\left(\sum_{N=3}^{L}\mathcal{E}_N[\varphi_1]^2\right)\;<\;\infty.$$This concludes the proof of Point 3 in Theorem~\ref{mainthm}.
\end{proof}

\section{Variations on the Theme of Randomisation in the Josephus Problem}\label{finalsec}

The main problem left open by Theorem~\ref{mainthm} is to determine the parameter of the Bernoulli limit distribution when $p\in (0, 1/3]\cup [2/3, 1)$; namely, with the notations of Claim (1) in the statement, to determine the probability of the event $\left\{X^{(p)}=1/2\right\}$. Denoting this parameter by $c(p)$, numerical simulations displayed in the appendix indicate that $c(p)=1$; in other words, this is saying that the sequence of random variables $\left(X_N^{(p)}\right)_{N\ge 3}$ converges in probability to the  constant $1/2$ regardless of the value of $p\in (0,1)$. Note that when $p\not\in (1/3, 2/3)$, the proof of Proposition~\ref{keypropopmidle}, which plays a crucial \emph{rôle} in establishing the second point in Theorem~\ref{mainthm}, is not valid anymore. The difficulty to determine the limit in this range can be gauged from the Central Limit Theorem stated in  Theorem~\ref{mainthm}~: even in the unbiased case $p=1/2$, the convergence in law towards the constant  $1/2$ is extremely slow (it is only of the order of the square root of a logarithm). \\

\noindent The following variant of the probabilistic elimination process provides more insight into  the subtlety involved in the determination of the existence of a limit measure.

\paragraph{\textbf{Alternative  Rule for the Probabilistic Elimination Process.}} \emph{Let there be $N$ participants enumerated from $0$ to $N-1$, standing on a unit circle with a regular spacing between them and labelled counterclockwise. The $0^{\textrm{th}}$  participant holds first the knife~: with probability $p$, he eliminates Participant 1 (standing on his right) and passes the knife onto the person to the right of the victim (namely, 2). Similarly, with probability $1-p$, he eliminates Participant $N-1$ (standing on his left) and passes  the knife onto the person to the left of the victim (namely, N-2). The next participant holding the knife then stabs the person still alive on his right-hand side with probability $p$ and the person still alive on his left--hand side with probability $1-p$, then passing the knife to the right and to the left, respectively.} \\


\noindent Compare with the first probabilistic elimination rule stated in the introduction~: the difference is that at each step, the outcome of the probabilistic choice is not about changing the \emph{direction} of the stabbing anymore but about choosing which of the persons standing on the left or on the right of the knife holder must be eliminated. \\


\noindent Denoting by $f_N(n,p)$ the probability that the $n^{th}$ participant should be the survivor in a process involving $N$ participants and evaluating the first argument $n$ modulo $N$, the recursion relation satisfied by the probability vector $\left(f_N(n,p)\right)_{0\le n\le N-1}$ can be shown to read

\begin{equation*}
f_N(n, p) = \begin{cases}
  p \cdot f_{N-1}(-1, p) + (1-p)\cdot  f_N(1, p)  &  \text{ if }  n \equiv 0 \pmod{N}  \\
  (1-p)\cdot f_{N-1}(2, p) & \text{ if } n \equiv 1 \pmod{N} \\
  p\cdot f_{N-1}(-2, p) & \text{ if } n\equiv -1\pmod{N} \\
  p \cdot f_{N-1}(n-2, p) + (1-p)\cdot  f_{N-1}(n+1, p)  &  \text{ otherwise }.
\end{cases}
\end{equation*}

\noindent The arguments developed in the previous sections enable one to show that the resulting sequence of random variables $\left(X_N^{(p)}\right)_{N\ge 3}$ (the notation with the previous rule is left unchanged for the sake of simplicity) converges in this case in probability to the constant $3p-1$ when $1/3<p<2/3$. Because the alternative rule of elimination under consideration coincides with the original one presented in the introduction in the unbiased case $p=1/2$, one also obtain a Central Limit Theorem in this situation. However, one does not obtain any conclusive statement about the limiting behavior of the sequence in the range $p\in (0, 1/3]\cup[2/3, 1)$. In fact, numerical simulations displayed in the appendix suggest that, in this range, the sequence does not admit any distributional limit anymore (in contrast with the first point in Theorem~\ref{mainthm}).\\

\noindent This second version of the elimination process can be seen as a particular case of an even more general probabilistic process involving two parameters $p, q \in [0,1]$~:

\paragraph{\textbf{General Rule for the Probabilistic Elimination Process.}} \emph{Let there be $N$ participants enumerated from $0$ to $N-1$, standing on a unit circle with a regular spacing between them and labelled counterclockwise.  Starting from the  $0^{\textrm{th}}$  participant, each one eliminates the participant on his right with probability p and the participant on his left with probability $1-p$; the knife is then passed onto the person on the right of the knife-holder with probability $q$ and onto the person onto the left of the knife-holder with probability $1-q$.} \\


\noindent Determining the distributional limit of the corresponding sequence of random variables as a function of the two parameters $p, q$ is an open problem. Some numerical simulations are provided in the appendix.


\bibliographystyle{unsrt}


\newpage

\section*{Appendix~: Numerical Simulations}\label{lctrootpole} 

\addcontentsline{toc}{section}{\protect\numberline{}Appendix~: Numerical Simulations}\label{lctrootpole} 
 
The three sets of graphs below represent numerical simulations in each of the three variants of the probabilistic Josephus problem considered so far, namely the first variant (referred to as R1) following the rule stated in the introduction, the second one (R2) with the alternative rule introduced in the final section  and the third one (R3) with the general rule depending on two parameters $p$ and $q$. The graphs have all been produced with simulations comprising $N=2000$ participants. They show the probability each of the 2000 players has of being the survivor. In particular, when the elimination process is deterministic (that is, when the parameters $p$ and $q$ take the extremal values 0 or 1), the survivor is uniquely determined by the number of participants. The graphs display in such cases a fixed value achieved with probability 1.

\begin{itemize}

\item \textbf{Numerical simulations for the probabilistic rule of elimination R1  with $N=2000$ persons.} The parameter $p$ takes successively the values 0, 0.2, 0.4, 0.6, 0.8 and 1. All the non-deterministic cases (i.e.~when $p\neq 0,1$) show the convergence in probability to the constant $1/2$.

\begin{figure}[ht!]
\begin{center}
\includegraphics[scale=0.4]{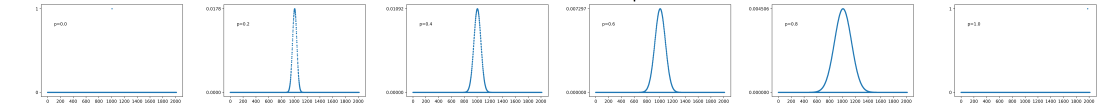}
\label{figure:}
\end{center}
\end{figure} 

\item \textbf{Numerical simulations for the alternative probabilistic rule of elimination R2  with $N=2000$ persons.} The parameter $p$ takes successively the values 0, 0.1, 0.2, 0.3, 0.4 and 0.5 (restricting the parameter $p$ to the interval $[0, 1/2]$ is without loss of generality~: the elimination process is indeed left unchanged upon swapping $p$ with $1-p$ and, correspondingly, the right and left moves). In the non-deterministic cases (i.e.~when $p\neq 0,1$), the graphs indicate a convergence of the process in probability to the constant $3p-1$ in the middle range $p\in (1/3, 2/3)$ and a divergence outside this range.

\begin{figure}[ht!]
\begin{center}
\includegraphics[scale=0.4]{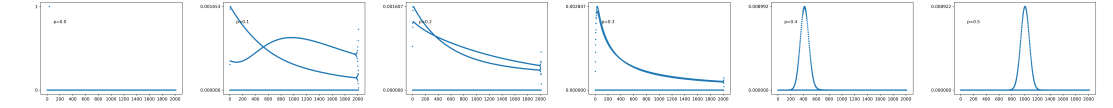}
\label{figure:}
\end{center}
\end{figure} 

\item \textbf{Numerical simulations for the general probabilistic rule of elimination R3 depending on two parameters $(p, q)$ with $N=2000$ players.} The distributional limit of the process displays here a more subtle dependency on the parameters $p$ and $q$ which is left to conjecture.

\begin{figure}[ht!]
\begin{center}
\includegraphics[scale=0.4]{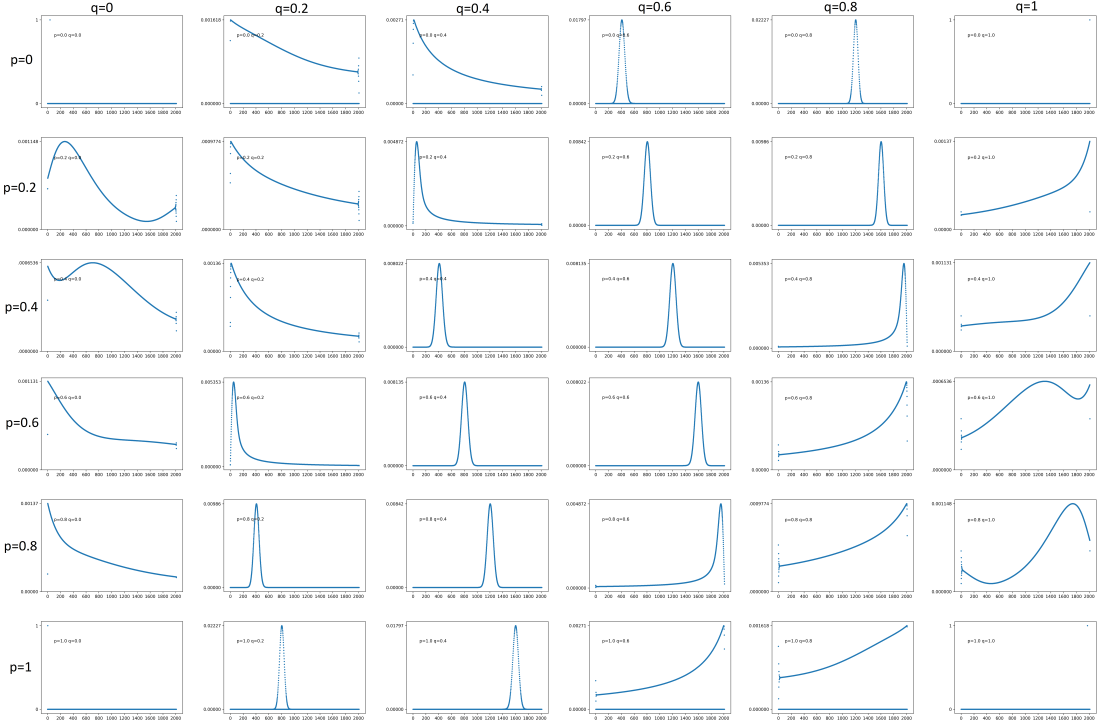}
\label{figure:}
\end{center}
\end{figure} 

\end{itemize}

\end{document}